\documentclass[a4paper,10pt]{amsart}
\usepackage[utf8]{inputenc}

\usepackage{amsmath,amsfonts,amsthm,amssymb,graphicx,xcolor} 
\usepackage{amsthm}
\usepackage{comment}

\usepackage{mathrsfs}
\usepackage{hyperref}

\usepackage{graphicx}
\usepackage{psfrag}
\usepackage{textcomp}

\newcommand{\p}{\partial}

\usepackage{epstopdf}

\newtheorem{Theorem}{Theorem}[section]

\newtheorem{Lemma}[Theorem]{Lemma}

\newtheorem{Proposition}[Theorem]{Proposition}

\newcommand{\s}{\hspace{0.5pt}}
\newcommand{\R}{\mathbb{R}}

\newcommand{\id}{\mathrm{Id}}

\renewcommand{\p}{\partial}

\usepackage[bordercolor=white, color=white, textwidth=100pt]{todonotes}

\title{Counterexamples to inverse problems for the wave equation}

\author[Liimatainen]{Tony Liimatainen}
\address{Department of Mathematics and Statistics, University of Jyv\"askyl\"a, Jyv\"askyl\"a, \newline \indent Finland 
\newline
 \indent Department of Mathematics and Statistics, University of Helsinki, Helsinki, Finland}
\curraddr{}
\email{tony.t.liimatainen@jyu.fi}

\author[Oksanen]{Lauri Oksanen}
\address{Department of Mathematics and Statistics, University of Helsinki, Helsinki, Finland}
\curraddr{}
\email{lauri.oksanen@helsinki.fi}

\begin{document}

\begin{abstract}
We construct counterexamples to inverse problems for the wave operator on domains in $\R^{n+1}$, $n \ge 2$, and on Lorentzian manifolds. We show that non-isometric Lorentzian metrics can lead to same partial data measurements, which are formulated in terms certain restrictions of the Dirichlet-to-Neumann map. The Lorentzian metrics giving counterexamples are time-dependent, but they are smooth and non-degenerate. On $\R^{n+1}$ the metrics are conformal to the Minkowski metric. 
\end{abstract}

\maketitle

\section{Introduction}
In this paper we construct counterexamples for inverse problems for the wave equation. Let us begin by describing our results on domains in $\R^{n+1}$ with $n \ge 2$.
Let $\Omega \subset \R^{n+1}$ be an open, bounded and connected domain with smooth boundary, and let $T > 0$.
We write 
    \begin{align}\label{def_M}
M = \overline \Omega \times [0,T], 
\quad \Sigma = \p\Omega \times (0,T),
    \end{align}
and consider the wave equation with a boundary value $u_0$,
\begin{equation}\label{eq:intro_wave-eq}
\begin{cases}
\qquad\qquad\qquad\square_g\s u = 0 &\text{in $M$}, \\
\qquad\qquad\qquad\quad \     u=u_0 &\text{on $\Sigma$}, \\
\quad \ \ \ \ \, \s u\big|_{t=0} = 0,\quad \partial_t u\big|_{t=0} = 0 &\text{on } \Omega.
\end{cases}
\end{equation}
Here $g$ is a smooth Lorentzian metric tensor on $M$, and the coordinate invariant wave operator associated to $g$ is given by 
\begin{equation}\label{invariant_wave_operator}
\square_gu=-\frac{1}{|g|^{1/2}}\p_a\left(|g|^{1/2}g^{ab}\p_bu\right).
\end{equation}
We assume that $\p_t$ is timelike for $g$, that is, 
    \begin{align*}
g(\p_t, \p_t) < 0.
    \end{align*}
We can also write $\square_g=\delta_gd$, where $d$ is the exterior derivative and $\delta_g$ is its formal $L^2$ adjoint with respect to volume form $dV_g$ induced by $g$. 

We write 
    \begin{align*}
\eta = \text{diag}(-1,1,\ldots,1)
    \end{align*}
for the Minkowski metric. 
Observe that $\square_\eta=\p_t^2-\Delta$, where $\Delta$ is the Laplacian on $\R^n$.
Our counterexamples in the case of domains will be given by metrics $g$ that are conformal to the Minkowski metric (or isotropic), that is, 
    \begin{align*}
g(x,t) = c(x,t) \eta
    \end{align*}
for a smooth, strictly positive function $c$.

Measurements on the lateral boundary $\Sigma$
are given by the Dirichlet-to-Neumann map (DN map), which is defined by the usual assignment,
	\begin{align*}
	\Lambda_g: C_0^\infty(\Sigma)\to C^\infty(\Sigma), \quad  \Lambda_g u_0 = \p_{\nu} u|_{\Sigma}.
	\end{align*}
	Here $\p_{\nu}$ denotes the normal derivative on $\Sigma$, defined with respect to $g$, and $u$ is the unique solution to~\eqref{eq:intro_wave-eq} with the boundary value $u_0$.
The below results are unaffected if $\Lambda_g$ is defined on a space of less regular function, say, on $H^s_0(\Sigma)$ with $s \ge 1$, or if data is given by the Neumann-to-Dirichlet map instead. 
	
	For open subsets $\Gamma_1$ and $\Gamma_2$ of $\Sigma$, we will consider the partial data DN map defined by
    \begin{align}\label{def_DN}
	\Lambda_g^{\Gamma_1,\Gamma_2}: C_0^\infty(\Gamma_1)\to C^\infty(\Gamma_2),
	\quad
	\Lambda_g^{\Gamma_1,\Gamma_2} u_0 = (\Lambda_g u_0)|_{\Gamma_2}.
    \end{align}
	 If $\Gamma_1\cap \Gamma_2=\emptyset$, we say that $\Lambda_g^{\Gamma_1,\Gamma_2}$ is a disjoint data DN map. 
From the physical point of view, sources and receivers are not used in the same place at the same time in the case of disjoint data. 

Let $\tilde g$ be another smooth Lorentzian metric tensor on $M$, and suppose that $\p_t$ is timelike also for $\tilde g$.
If $\phi : M \to M$ is a diffeomorphism satisfying 
    \begin{align*}
\phi^* \tilde g= g 
\quad\text{and}\quad
\phi|_{\Gamma_1 \cup \Gamma_2} = \id,
    \end{align*}
then $\Lambda_{\tilde g}^{\Gamma_1,\Gamma_2} = \Lambda_g^{\Gamma_1,\Gamma_2}$.
In other words, the data $\Lambda_g^{\Gamma_1,\Gamma_2}$ is invariant under isometries fixing $\Gamma_1 \cup \Gamma_2$.
We will construct counterexamples to the inverse problem to determine $g$, modulo this invariance, given $\Lambda_g^{\Gamma_1,\Gamma_2}$.

We say that a hyperplane in $\R^{n+1}$
is lightlike if its normal vector with respect to the Minkowski metric
is lightlike. More explicitly, a lightlike hyperplane is defined by the equation
	  \[
	   (x-x_0)\cdot \theta-(t-t_0)=0
	  \]
for some fixed $(x_0,t_0)\in \R^{n+1}$ and 
a unit vector
$\theta\in \R^n$ with respect to the Euclidean metric of $\R^n$. 
	 
	\begin{Theorem}[Partial data on a domain]\label{thm_Rn}
	 Let $\Gamma\subset \Sigma$ be open
	 and suppose that there is a lightlike hyperplane that is intersecting the interior of $M$ but not the closure of $\Gamma$.
Then there is an infinite family $\mathcal G$ of smooth Lorentzian metrics on $M$ such that for all $g \in \mathcal G$ there holds
	 \begin{equation}\label{identical_partial_data_Rn}
	  \Lambda_g^{\Gamma,\Gamma}=\Lambda_{\eta}^{\Gamma,\Gamma}.
	 \end{equation}
Here $g$ is not isometric to $\eta$, that is, there is no diffeomorphism $\phi : M \to M$ satisfying $\phi^* \eta = g$ and $\phi|_\Gamma = \id$.
Moreover, $g = c \eta$ for some conformal factor $c$,
and the family $\mathcal G$ is not bounded in the $L^\infty$-norm. 
	\end{Theorem}

It should be emphasized that the data $\Lambda_g^{\Gamma,\Gamma}$ is not invariant under conformal transformations taking $g$ to $cg$ with strictly positive $c$ satisfying $c=1$ near $\Gamma$.
In fact, if $\Gamma_0 \subset \p \Omega$ is open and nonempty, $\Gamma = (0,T) \times \Gamma_0$ for large enough $T > 0$, and $c$ is independent of time (and satisfies the two assumptions above),
then $c$ can be recovered given 
$\Lambda_{c\eta}^{\Gamma,\Gamma}$,
see Lemma \ref{lem_timeindep} in the appendix for the details. 

Let us now turn to our result in the case of a smooth, connected Lorentzian manifold with boundary $(\mathcal M,g)$.
We assume that the dimension $n+1$ of $\mathcal M$ satisfies $n \ge 2$. Following \cite{Alexakis2020} we assume, furthermore, that 
\begin{itemize}
\item[(i)] The boundary $\p \mathcal M$ is timelike.
\item[(ii)] There is a smooth, proper, surjective temporal function $\tau : \mathcal M \to \R$.
\end{itemize}
The boundary being timelike means that
$g(\p_\nu, \p_\nu) > 0$,
where $\p_{\nu}$ is the normal derivative on $\p \mathcal M$,
and $\tau$ being temporal means that $g(d\tau, d\tau) < 0$.
Proper is used in the topological sense, that is, inverse images of compact
subsets are compact under $\tau$.
In this case, we let $T>0$ and write 
    \begin{align*}
M = \tau^{-1}([0,T]),
\quad \Sigma = \p \mathcal M \cap \tau^{-1}((0,T)),
\quad \Omega = \tau^{-1}(\{0\}),
    \end{align*}
where $\tau^{-1}(S) = \{p \in \mathcal M : \tau(p) \in S\}$
for $S \subset \mathcal M$.
The analogue of (\ref{eq:intro_wave-eq}) is
\begin{equation}\label{eq:intro_wave-eq_lorentz}
\begin{cases}
\qquad\qquad\qquad\square_g\s u = 0 &\text{in $M$}, \\
\qquad\qquad\qquad\quad \     u=u_0 &\text{on $\Sigma$}, \\
\quad \ \ \ \ \, \s u\big|_{\tau=0} = 0,\quad \partial_\tau u\big|_{\tau=0} = 0 &\text{on } \Omega,
\end{cases}
\end{equation}
where $\p_\tau u = g(d\tau, du)$.
It follows from \cite[Th. 24.1.1]{H3} that (\ref{eq:intro_wave-eq_lorentz}) has a unique solution $u \in H^s(M)$ when $u_0 \in H_0^s(\Sigma)$ and $s \ge 1$.
The strict hyperbolicity needed for this theorem follows from $\tau$ being temporal, see Lemma \ref{lem_strict_hyper} in the appendix for the details. 

For open subsets $\Gamma_1$ and $\Gamma_2$ of $\Sigma$, we can again define the partial data DN map by (\ref{def_DN}),
where $u$ is now the solution to (\ref{eq:intro_wave-eq_lorentz}). In contrasts to the case of a domain, we can construct counterexamples only for disjoint data in the Lorentzian case.

	\begin{Theorem}[Disjoint data on a Lorentzian manifold]\label{thm_Lorentz}
Let $\Gamma_1,\Gamma_2$ be open subsets of $\Sigma$ satisfying $\Gamma_1\cap\Gamma_2=\emptyset$ and $\Gamma_1 \cup \Gamma_2 \ne \Sigma$. 
Then there is an infinite family $\mathcal G$ of Lorentzian metrics on $M$ such that for all $\tilde g \in \mathcal G$ there holds
	 \begin{equation}\label{identical_partial_data_Lorentz}
	  \Lambda_{\tilde g}^{\Gamma_1,\Gamma_2}=\Lambda_{g}^{\Gamma_1,\Gamma_2}.
	 \end{equation}
Here $\tilde g$ is not isometric to $g$, that is, there is no diffeomorphism $\phi : M \to M$ satisfying $\phi^* g = \tilde g$ and $\phi|_{\Gamma_1 \cup \Gamma_2} = \id$.
Moreover, $\tilde g = c g$ for some conformal factor $c$.
	\end{Theorem}
	
Both Theorems \ref{thm_Rn} and \ref{thm_Lorentz}
are based on the fact that the wave operator has the following ``hidden conformal invariance''
	\[
  \square_{f^{p-2}g}(f^{-1}u)=f^{1-p}\square_gu,
	\]
	if $f$ is a positive solution to $\square_gf=0$, see Proposition~\ref{hidden_conformal_invariance} below. Here $p$ is the constant 
    \begin{align}\label{def_p}
p = 2\, \frac{n+1}{n-1},
    \end{align}
and we recall the standing assumption that $\dim(M) = n+1$
with $n \ge 2$. 

The Laplace--Beltrami operator on a Riemannian manifold has the same  hidden conformal invariance, and the Riemannian analogue of Theorem~\ref{thm_Lorentz} was proven in~\cite{DKN17}. However, there appears to be no results similar to Theorem~\ref{thm_Rn} in the Riemannian case. The reason for this is that the wave equation in the Minkowski geometry admits non-trivial solutions supported close to a lightlike hyperplane, but solutions to the Laplace equation can not vanish in an open, nonempty set without vanishing everywhere.
To our knowledge, the only nonuniqueness result with partial data analogous to $\Gamma_g^{\Gamma, \Gamma}$ in the Riemannian case is \cite{Daude2020}, but contrary to Theorem~\ref{thm_Rn}, the counterexamples in \cite{Daude2020} are not smooth. They are smooth in the interior of the manifold and H\"older continuous up to the boundary.
We mention also \cite{Daude2019} for counterexamples related to \cite{DKN17} and the review \cite{Daude2018review} of nonuniqueness results in the Riemannian case. For up-to-date results on the borderline between uniqueness and non-uniqueness in the Riemannian case we refer to \cite{Daude2020a}.

It is likely that black hole type spacetimes can give non-smooth counterexamples to the inverse problem for the wave equation, but this question has not been systematically studied to our knowledge. 
In the Riemannian case, singular counterexamples have been extensively studied by using transformation optics. This type of counterexamples are often called invisibilty cloacking. 

From the physical point of view, invisibility cloaking means covering an object with a special material so that the light or other electromagnetic waves go around the object. This creates an illusion that in the place where the object is located, there is only homogeneous background space, for example, air or vacuum. 
From the mathematical point of view, invisibility cloaking by using transformation optics was first studied in \cite{Greenleaf2003} where counterexamples to the inverse problem for the Laplace--Beltrami operator were considered. In dimension $2$, where the Laplace--Beltrami operator is conformally invariant, first counterexamples were constructed in~\cite{LTU} by using a conformal blow up construction. Also the counterexamples of~\cite{LTU} are singular. Singular conductivities are not physical, but approximate realizations of such conductivities, and thus invisibility cloaking, have been implemented by using metamaterials, see~\cite{Pendry, PPS}.  We refer to the review~\cite{Uhlmann} for more on invisibility cloaking. 

The counterexamples we give for the inverse problem for the wave operator are smooth and physically reasonable. In fact, our examples in the case of a domain are similar to gravitational waves in the Minkowski space. 

Let us finish by mentioning some positive results regarding inverse problems for the wave equation. Positive results with partial data are mostly confined to the case of time independent coefficients,
but \cite{Eskin2007} and \cite{Kian2019a} contain partial data results with time dependent coefficients. 
In the case of time independent coefficients, knowledge of the partial data DN map for an arbitrary open subset of the boundary is sufficient to determine the coefficients of the wave equation (up to natural gauge symmetries)~\cite{KKL}, and results with disjoint data are proven in~\cite{Kian2019,LOa,LOb}. In~\cite{Rakesh_2000} it is shown that certain time independent coefficients of the $1+1$-dimensional wave equation are determined by disjoint data measurements.

In the time independent case,
knowledge of the DN map for $\p_t^2-\Delta_{g}$ is equivalent to knowing the DN map for $\Delta_g-\lambda^2$ at all frequencies $\lambda$. Due to this correspondence, we also mention the work~\cite{KS}, which contains a review of the known results for the partial data Calder\'on problem on a Riemannian manifold.

Finally we mention the work~\cite{LLS} concerning the Calder\'on problem for the conformal Laplacian that served as a motivation for this work. It also contains further discussion of invisibility cloaking.

\vspace{10pt}

\noindent {\bf Acknowledgments.} 
T. L. and L. O. were supported by the Finnish Centre of Excellence in Inverse Modelling and Imaging, Academy of Finland grant 284715.
L. O. was supported by EPSRC grants EP/R002207/1 and EP/P01593X/1.

\section{Hidden conformal invariance}
The Lorentzian wave operator~\eqref{invariant_wave_operator} is invariant under conformal scaling of the Lorentzian metric in the dimension $2$, that is, in the case $\dim(M)=1+1$. 
This is no longer true in higher dimensions for an arbitrary rescaling. However, if the conformal factory satisfies $\square_gf=0$, we have the following invariance.

 \begin{Proposition}\label{hidden_conformal_invariance}
  Let $f>0$ be a positive function. Then 
  \[
  \square_{f^{p-2}g}(f^{-1}u)=f^{1-p}\s\square_gu-u f^{-p}\s\square_gf,
  \]
  where $p$ is given by (\ref{def_p}).
In particular, 
  \[
   \square_{f^{p-2}g}(f^{-1}u)=f^{1-p}\s \square_gu
  \]
  for all smooth functions $u$ if and only if $\square_gf=0$. Also, if $\square_gf=0$ and $\square_gu=0$, then $\square_{f^{p-2}g}(f^{-1}u)=0$.
 \end{Proposition}
  Let us write
\[
 m=\dim(M)=n+1, \quad \alpha_m = \frac{m-2}{4(m-1)}.
\]
Proposition~\ref{hidden_conformal_invariance} follows from considerations of the conformal wave operator
\[
 \mathcal{L}_g = \square_g + \alpha_m R(g),
\]
where $R(g)$ is the scalar curvature of $g$, acting by multiplication on functions. The conformal wave operator is conformally invariant in the following sense
\begin{equation}\label{conformal_invariance_of_conformal_wave_operator}
 \mathcal{L}_{f^{p-2}g}f^{-1}u=f^{1-p}\mathcal{L}_g u.
\end{equation}
We will also need the fact that the scalar curvature changes under conformal scalings as
\begin{equation}\label{scalar_curvature_conformal_scaling}
 R(f^{p-2}g)=f^{1-p}\left(\alpha_m^{-1} \square_gf+R(g)f\right).
\end{equation}
For (\ref{conformal_invariance_of_conformal_wave_operator})--(\ref{scalar_curvature_conformal_scaling}), see e.g.~\cite{CG18,PL}. We are ready to prove Proposition~\ref{hidden_conformal_invariance}.
\begin{proof}[Proof of Proposition~\ref{hidden_conformal_invariance}]   
We have by using (\ref{conformal_invariance_of_conformal_wave_operator})--(\ref{scalar_curvature_conformal_scaling}),
    \begin{align*}
\square_{f^{p-2}g}(f^{-1}u)
&=
\mathcal{L}_{f^{p-2}g}(f^{-1}u)-\alpha_m R(f^{p-2}g)f^{-1}u
\\&=
f^{1-p} \left(\mathcal{L}_g u - \alpha_m (\alpha_m^{-1} \square_gf+R(g)f) f^{-1}u \right)
\\&=
f^{1-p} \Box_g u - u f^{-p} \Box_g f.
    \end{align*}
\end{proof}

Of course, Proposition~\ref{hidden_conformal_invariance} can be proved without referring to the conformal wave operator and scalar curvature. (However, figuring out the invariance exists without using the conformal wave operator might be less straightforward.) Slightly less obvious is the fact that the  conformal scaling invariance of Proposition~\ref{hidden_conformal_invariance}, is not a feature of the invariant wave operator, but a feature of general second order differential operators in divergence form. We include another proof of Proposition~\ref{hidden_conformal_invariance}, which we give in a more general setting. We consider second order differential operators $P_C:C^\infty(\R^m, \R^k)\to C^{\infty}(\R^m, \R^l)$, acting thus possibly on vector fields, given in divergence form
 \begin{equation}\label{general_form_operator_PC}
   (P_Cu)_j(x)=\sum_{a,b=1}^{m}\sum_{c=1}^{k}\p_a \big(C_{jabc}(x)\p_b u_c(x)\big), 
  \end{equation}
  where $j=1,\ldots, l$. Here $u=(u_1,u_2,\ldots,u_k)$ is a vector field $\R^m\to \R^k$ and $C_{jabc}(x)$, $a,b=1,\ldots m$, $c=1,\ldots, k$, is the coefficient matrix field of $P_C$. We will abuse the notation slightly and denote by $P_C$ also the operator $C^\infty(\R^m)\to C^\infty(\R^m ,\R^{lk})$ given by 
  \[
(P_Cf)_{jc}=\sum_{a,b=1}^{m}\p_a (C_{jabc}\p_b f). 
\]
 \begin{Proposition}
  Let $P_C$ be a second order differential operator on $\R^m$ acting on (possibly) vector valued functions as in~\eqref{general_form_operator_PC} 
  with $m,l,k\geq 1$. We assume the coefficient matrix field $C_{jabc}$ of $P_C$ to be symmetric under the change of the indices $a$ and $b$, $C_{jabc}=C_{jbac}$
  
  Let $f>0$ be a positive function. Then $P$ satisfies the following formula
  \[
 P_{f^2C}(f^{-1}u)_j=f P_{C}u_j-u \cdot P_Cf_j,
\]
where $u \cdot P_Cf_j=\sum_{c=1}^ku_c(P_Cf)_{jc}$.
If $k=l=1$ and $(C_{jabc})=(\sigma_{ab})$, then we especially have
 \[
  \nabla\cdot f^2 \sigma \nabla (f^{-1}u)=f\s \nabla\cdot \sigma \nabla u -u \s \nabla\cdot \sigma \nabla f.
 \]
  \end{Proposition}
  \begin{proof}
   Let $f$ be a smooth positive function. We use Einstein summation over repeated indices and calculate
  \begin{align*}
 P_{f^2C}(f^{-1}u)_j&=\p_a \big(f^2 C_{jabc}\p_b(f^{-1}u_c)\big)=\p_a\big(f C_{jabc}\p_b u_c\big) + \p_a \big(f^2 (-f^{-2}\p_b f) C_{jabc}u_c\big)  \\
 &=(\p_a f)C_{jabc}(\p_b u_c) + f P_Cu_j - \p_a \big(u_c C_{jabc}\p_b f)\big) \\ 
 &=(\p_a f)C_{jabc}(\p_b u_c) + f P_Cu_j- (\p_au_c)C_{jabc}(\p_b f)-u_c (P_Cf)_{jc} \\
 &=f P_Cu_j- u \cdot P_Cf_j,
\end{align*}
where in the last equality we used $C_{jabc}=C_{jbac}$.
\end{proof}
%
 
\section{Rigidity of conformal transformations} 

In this section, $(M,g)$ is a smooth connected Lorentzian manifold with boundary, and we show that the identity map is the only conformal mapping on $M$ fixing an open subset of the boundary. Diffeomorphisms in this section are assumed to be diffeomorphic up to boundary.

\begin{Proposition}\label{lem_conf_triv}
Let $\Gamma \subset \p M$ be open, nonempty and timelike.
Let $\phi$ be a conformal diffeomorphism of $(M,g)$ that satisfies
$\phi|_{\Gamma} = \id$ and $\phi^* g = e^{2h} g$ for $h \in C^\infty(M)$.
Then $\phi = \id$.
\end{Proposition}

Timelike could be replaced by spacelike in the proposition. 
With the natural replacements, the proposition holds also for Riemannian manifolds with boundary. In the case of a Riemannian manifold the proposition was proven in~\cite{Li} (with $\Gamma = \p M$) and more generally for conformal mappings between different Riemannian manifolds $(M,g)$ and $(\tilde M, \tilde g)$ in \cite{LL}. 
The work~\cite{LL} also contains a related uniqueness result for Lorentzian conformal mappings corresponding to the case that $\Gamma$ is a Cauchy surface.  

\begin{Lemma}
Let $\phi : (M,g) \to (\tilde{M},\tilde g)$ be a diffeomorphism satisfying 
$\phi^* \tilde g = e^{2h} g$ for $h \in C^\infty(M)$.
We have in local coordinates
    \begin{align}\label{de_dphi}
\frac{\p^2\phi^{\s j}}{\p x^k \p x^l}
= 
\mathcal N^j_{kl}(d\phi, dh) + 
\frac{\p\phi^j}{\p x^a} \Gamma^a_{kl}(g) 
- \frac{\p\phi^a}{\p x^k} \frac{\p\phi^{\s b}}{\p x^l}
\Gamma^j_{ab}(\tilde g) \circ \phi,
    \end{align}
where
\[
 \mathcal N^j_{kl}(d\phi, dh)=\frac{\p \phi^j}{\p x^l}\frac{\p h}{\p x^k}
+ \frac{\p \phi^j}{\p x^k}\frac{\p h}{\p x^l} 
- g^{ab}\frac{\p \phi^j}{\p x^a} \frac{\p h}{\p x^b} g_{kl},
\]
satisfies $\mathcal{N}^j_{kl}(d\phi,0)=0$. Here $\Gamma^a_{kl}(g)$ and $\Gamma^j_{ab}(\tilde g)$ are the Christoffel symbols of $g$ and $\tilde{g}$ respectively in local coordinates. 
\end{Lemma}
\begin{proof}
We write $\tilde x = \phi(x)$ and use the typical convention to denote by $\p x^i/\p \tilde{x}^m$ the components of the differential of the inverse of $\phi$ (evaluated at $\phi$). We then have the standard formulas 
    \begin{align}\label{de_dphi_aux1}
\Gamma^i_{kl}(\phi^*\tilde g) =
\frac{\partial x^i}{\partial \tilde x^m}\,
\frac{\partial \tilde x^a}{\partial x^k}\,
\frac{\partial \tilde x^b}{\partial x^l}\,
\Gamma^m_{ab}(\tilde g) \circ \phi
+ 
\frac{\partial^2 \tilde x^m}{\partial x^k \partial x^l}\,
\frac{\partial x^i}{\partial \tilde x^m},
\end{align}
 and 
    \begin{align}\label{de_dphi_aux2}
\Gamma^i_{kl}(e^{2h} g) =
\Gamma^i_{kl}(g)
+ \frac{\p h}{\p x^k} \delta^i_l
+ \frac{\p h}{\p x^l} \delta^i_k
- g^{ia} \frac{\p h}{\p x^a} g_{kl},
    \end{align}
see e.g.~\cite{Be07} for the latter formula.
The assumption $\phi^* \tilde g = e^{2h} g$ implies that $\Gamma^i_{kl}(\phi^*\tilde g) = \Gamma^i_{kl}(e^{2h} g)$, and the claimed formula follows after multiplying (\ref{de_dphi_aux1}) and (\ref{de_dphi_aux2}) by 
$\p\tilde x^j /\p x^i$.
\end{proof}

Recall, that we denote the dimension $n+1$ of $M$ by $m$.

\begin{Lemma}\label{transformation_of_Shouten}
Let $\phi : (M,g) \to (\tilde M,\tilde g)$ be a diffeomorphism satisfying 
$\phi^* \tilde g = e^{2h} g$ for $h \in C^\infty(M)$.
We have in local coordinates
    \begin{align}\label{de_dh}
\frac{\p^2 h}{\p x^k \p x^l}
= \mathcal M_{kl}(dh) + P_{kl}(g) - \frac{\p\phi^a}{\p x^k} \frac{\p\phi^b}{\p x^l} {P}_{ab}(\tilde g) \circ \phi,
    \end{align}
    where 
    \[
     \mathcal M_{kl}(dh)=\p_k h\p_l h-m g(dh,dh)\s g_{kl}
    \]
satisfies $\mathcal M_{kl}(0) = 0$.
Here $P_{ab}(g)$ is the Schouten tensor of $g$ defined in terms of the Ricci tensor $R_{ab}$ and scalar curvature $R$ as
    \begin{align*}
 P_{ab}(g) =
\frac{1}{m-2}\big(R_{ab}(g)- \frac{1}{2(m-1)} R(g) g_{ab}\big).
    \end{align*}
%
\end{Lemma}
\begin{proof}
We write $q(\xi) = g(\xi,\xi)$.
Then we have the formulas
    \begin{align*}
R_{ab}(e^{2h} g) =
R_{ab}(g)
-(m-2)(\p_a \p_b h - \p_a h \p_b h)
-(\Box_g h + (m-2) q(dh)) g_{ab},
    \end{align*}
see e.g.~\cite{PL}. Consequently, we obtain
    \begin{align*}
R(e^{2h} g) (e^{2h} g)_{kl} 
&=
R(g)g_{kl}
-(m-2)(\Box_g h - q(dh)) g_{kl}
-m(\Box_g h + (m-2) q(dh)) g_{kl}
\\&=
R(g)g_{kl}
-2(m-1)(\Box_g h) g_{kl} - (m-1)(m-2) q(dh) g_{kl}.
    \end{align*}
Hence
    \begin{align*}
P_{kl}(e^{2h} g) 
= 
P_{kl}(g) 
-(\p_k \p_l h - \p_k h \p_l h)
-m q(dh) g_{kl}.
    \end{align*}
The Schouten tensor is a covariant two tensor, and hence
    \begin{align*}
P_{kl}(\phi^*\tilde g) = \frac{\p\phi^a}{\p x^k} \frac{\p\phi^b}{\p x^l} P_{ab}(\tilde g) \circ \phi.
    \end{align*}
The claim follows from $P_{ab}(\phi^*\tilde g) = P_{ab}(e^{2h} g)$.
\end{proof}

We have the following variation of \cite[Lemma 4.1]{Li}. 
\begin{Lemma}\label{lem_conf_triv_aux}
Let $\Gamma \subset \p M$ be open, nonempty and timelike.
Let $\phi$ be a conformal diffeomorphism of $(M,g)$
satisfying
$\phi|_{\Gamma} = \id$ and $\phi^* g = e^{2h} g$ for $h \in C^\infty(M)$.
Then the 1-jets of $\phi$ and $\id$ coincide at all points on $\Gamma$. Moreover, $h = 0$ and $dh = 0$ on $\Gamma$.
\end{Lemma}
\begin{proof}
We express $g$ and $\phi$ in boundary normal coordinates $(r,y)$ where $y$ are local coordinates on $\Gamma$ and where $\{r=0\}$ corresponds to $\Gamma$. There holds
    \begin{align*}
g = dr \otimes dr + h_{jk} dy^j \otimes dy^k,
    \end{align*}
and $\phi|_{\Gamma} = \id$ implies
    \begin{align}\label{T_exp_of_phi}
\phi^0 = a(y)r + \mathcal O(r^2),
\quad
\phi^j = y^j + b^j(y)r + \mathcal O(r^2).
    \end{align}
Moreover, by using the identity $\phi^*(df)=d(\phi^*f)$, valid for any $f$, we have
    \begin{align*}
\phi^* g = d\phi^0 \otimes d\phi^0 + (h_{jk} \circ \phi) d\phi^j \otimes d\phi^k.
    \end{align*}
On the other hand, by using~\eqref{T_exp_of_phi}, we obtain the formula
    \begin{align*}
\phi^* g = a^2 dr \otimes dr + h_{jk} (
b^j b^k dr \otimes dr + 
b^j dr \otimes dy^k +  
b^k dy^j \otimes dr +  
dy^j \otimes dy^k),
    \end{align*}
    which holds on $\Gamma$. We write $\lambda = e^{2h}$ and have
    \begin{align*}
\phi^* g = \lambda g = \lambda dr \otimes dr + \lambda h_{jk} dy^j \otimes dy^k.
    \end{align*}
By comparing the $dy^j \otimes dy^k$ components, we see that $\lambda = 1$ on $\Gamma$ since $\Gamma$ is not lightlike. Moreover, the
$dr \otimes dy^k$ components give $b^j = 0$, and $dr \otimes dr$ component give $a=1$. Thus the $1$-jet of $\phi$ coincides with that of the identity mapping on $\Gamma$.

It remains to show that $dh = 0$ on $\Gamma$. By using $a(y)=1$ and $b^j(y)=0$, we rewrite~\eqref{T_exp_of_phi} as
    \begin{align*}
\phi^0 = r+ \frac12 \tilde a(y)r^2 + \mathcal O(r^3),
\quad
\phi^j = y^j + \frac12 \tilde b^j(y)r^2 + \mathcal O(r^3).
    \end{align*}
Then
    \begin{align*}
\phi^* g
=
(1+2\tilde a r) dr \otimes dr
+ h_{jk}(r\tilde b^j dr \otimes dy^k +  
r \tilde b^k dy^j \otimes dr +  
dy^j \otimes dy^k) + \mathcal O(r^2)
    \end{align*}
holds on $\Gamma$. By writing $\lambda = 1 + c(y) r + \mathcal O(r^2)$, we have
    \begin{align*}
\lambda g = (1+cr) dr \otimes dr + (1+cr) h_{jk} dy^j \otimes dy^k + \mathcal O(r^2).
    \end{align*}
By comparing the $dy^j \otimes dy^k$ components, we see that $c = 0$ on $\Gamma$.  Thus $dh|_\Gamma=0$.
\end{proof}


We are now ready to prove Proposition~\ref{lem_conf_triv}.
\begin{proof}[Proof of Proposition \ref{lem_conf_triv}]
Let $z\in M$, and let us show that $\phi(z)=z$. For this let $y\in \Gamma$ and let $\gamma:[0,1]\times M$ be smooth path on $M$ such that $\gamma(0)=y$ and $\gamma(1)=z$. Let us define
\[
 B\subset [0,1]
\]
as the largest interval of points $s\in [0,1]$ containing zero such that $\phi(\gamma(s))=\gamma(s)$, $d\phi$ at $\gamma(s)$ is the identity on $T_{\gamma(s)}M$, and $dh = 0$ at $\gamma(s)=0$. We will show that $B=[0,1]$, which then proves $\phi(z)=z$ and consequently the claim of the proposition.

Let $(x^0, \dots, x^n)$ be local coordinates in a neighborhood $U$ of $y$.
It follows from $\phi(y) = y$ that both $\gamma(s)$ and $\phi(\gamma(s))$ are in $U$ or small $s > 0$.
We evaluate the formulas~\eqref{de_dphi} and~\eqref{de_dh} at $\gamma(s)$ with $\tilde g=g$, and obtain in the fixed local coordinates 
    \begin{align}\label{de_dphi_proof}
\frac{\p^2\phi^{\s j}}{\p x^k \p x^l}\Big|_{\gamma(s)}
&= 
\left(\mathcal N^j_{kl}(d\phi, dh)+ 
\frac{\p\phi^j}{\p x^c} \Gamma^c_{kl}(g)\right)\Big|_{\gamma(s)} 
- \left(\frac{\p\phi^a}{\p x^k} \frac{\p\phi^{\s b}}{\p x^l}\right)\Big|_{\gamma(s)}
\Gamma^j_{ab}(g)\Big|_{\phi(\gamma(s))},
\\\label{de_dh_proof}
\frac{\p^2 h}{\p x^k \p x^l}\Big|_{\gamma(s)}
&= \left(\mathcal M_{kl}(dh)+ P_{kl}(g)\right)\Big|_{\gamma(s)} - \left(\frac{\p\phi^a}{\p x^k}\frac{\p\phi^b}{\p x^l}\right)\Big|_{\gamma(s)} {P}_{ab}(g)\Big|_{\phi(\gamma(s))}.
    \end{align}
    
Let us define a matrix field $X$ over $\gamma$ and vector fields $Y$ and $Z$ over $\gamma$ by
\begin{align*}
 X^j_l(s)=\frac{\p \phi^j}{\p x^l}\circ \gamma(s), \quad
 Y_l(s)=\frac{\p h}{\p x^l}\circ \gamma(s), \quad
 Z^j(s)=\phi^j\circ \gamma(s).
\end{align*}
Observe that
\[
\p_sX^j_l(s)=\p_s\Big(\frac{\p \phi^j}{\p x^l}\circ \gamma(s)\Big)=\frac{\p^2 \phi^j}{\p x^k \p x^l}\Big|_{\gamma(s)}\dot\gamma^k(s),
\]
and that analogous formulas hold for $\p_s Y_l$ and $\p_s Z^j$.
Contracting~\eqref{de_dphi_proof} and~\eqref{de_dh_proof} with $\dot\gamma^k(s)$ gives 
    \begin{align}\label{ODE_system}
\begin{cases}
\p_sX^j_l
= 
\tilde{\mathcal N}^j_{l}(X,Y) + 
\dot\gamma^kX^j_c\Gamma^c_{kl}(g) \circ\gamma 
- \dot\gamma^k X_k^a X^b_l
\Gamma^j_{ab}(g) \circ Z,
\\ 
\p_sY_j
= 
\tilde{\mathcal M}_{j}(Y)
+ \dot\gamma^k P_{kj}(g) \circ \gamma 
- \dot\gamma^k X_j^b X_k^c {P}_{bc}(g) \circ Z,
\\
\p_sZ^j
=
\dot\gamma^k X^j_k,
\end{cases}
    \end{align}
where 
$\tilde{\mathcal N}^j_{l} := \dot\gamma^k(s)\mathcal N^j_{kl}|_{\gamma(s)}$
and
$\tilde{\mathcal M}_{j} := \dot\gamma^k(s)\mathcal M_{kj}|_{\gamma(s)}$.
We see that $(X,Y,Z)$ solves a system of first order ordinary differential equations.

Let us define 
\begin{equation}\label{idential_sol_to_system}
 \overline{X}(s)=I_{m}\in \R^{m\times m}, \quad \overline{Y}(s)=0\in \R^m, \quad \overline{Z}(s)=\gamma(s)\in \R^m,
\end{equation}
where $I_{m}$ is the identity matrix. 
For $X = \overline X$ and $Z = \overline Z$, the last two terms in the first equation in (\ref{ODE_system}) cancel out,
and the same holds for the last two terms in second equation.
Moreover, with the above choice, $\p_s Z^j = \dot \gamma^j = \dot\gamma^k X^j_k$.
Finally, for $Y = 0$, there holds
$\mathcal{N}^j_{kl}(X,Y)=0$ and $\mathcal M_{kl}(Y)=0$,
and we see that $(\overline{X},\overline{Y},\overline{Z})$ is a solution to the system~\eqref{ODE_system}. We have also
    \begin{align}\label{de_init}
(\overline{X}(0),\overline{Y}(0),\overline{Z}(0))=(d\phi\circ\gamma(0),dh\circ\gamma(0),\phi\circ \gamma(0)),
    \end{align}
since by Lemma~\ref{lem_conf_triv}, at $\gamma(0)=y\in \Gamma$, the $1$-jet of $\phi$ is that of the identity map of $M$ and $dh=0$. 
Since the quantities on the right hand sides of the equations in \eqref{ODE_system} depend Lipschitz continuously on $(X,Y,Z)$, this system has a unique solution on the interval $[0,\delta)$ for some $\delta >0$ given the initial data (\ref{de_init}).
Therefore 
    \begin{align*}
(X(s),Y(s),Z(s))=(\overline{X}(s),\overline{Y}(s),\overline{Z}(s)), \quad s \in [0,\delta). 
    \end{align*}
We have shown that $[0,\delta)\subset B$.
Letting $s \in B$, replacing $y$ by $\gamma(s)$,
and repeating the above argument, shows that $B$ is open as a subset of $[0,1]$.

Let us show that $B$ is closed. 
Let $s_j \in B$, $j=1,2,\dots$, and suppose that $s_j$ converges to a point $s \in [0,1]$.
We need to show that $s \in B$.
By continuity 
    \begin{align*}
\phi(\gamma(s)) 
= \lim_{j \to \infty} \phi(\gamma(s_j))
= \lim_{j \to \infty} \gamma(s_j) = \gamma(s). 
    \end{align*}
Analogously, by continuity, $dh = 0$ at $\gamma(s)=0$.
We choose local coordinates in a neighborhood $U$ of $\gamma(s)$. Then $\gamma(s_j) \in U$ for large $s_j$,
and $d\phi|_{\gamma(s_j)} = I_m$ in the coordinates.
By continuity, $d\phi|_{\gamma(s)} = I_m$ in the coordinates.
We have shown that $B$ is closed.
By connectedness it follows that $B=[0,1]$, which concludes the proof.
%
%
%
%
%
%
\end{proof}
 
\section{Construction of counterexamples}

We consider first the case of a domain in $\R^{n+1}$, and make use of explicit solutions to $\square_\eta f=0$. Let $\theta\in \R^n$ be a unit vector (with respect to the Euclidean length), $x_0\in \R^n$ and $t_0\in \R$, and let $H\in C_0^\infty(\R)$ be a non-negative function. Then the function
 \begin{equation}\label{formula_f_Rn}
  f(t,x)=1+H((x-x_0)\cdot \theta -(t-t_0))
 \end{equation}
 is a solution to $\Box_\eta f = 0$ in $\R^{n+1}$. 
 
\begin{proof}[Proof of Theorem~\ref{thm_Rn}]
Let $\Gamma \subset \Sigma$ be an open set whose closure does not intersect the hyperplane $X$ defined by
 \[
  (x-x_0)\cdot \theta -(t-t_0)=0.
 \]
As $X$ does not intersect $\Gamma$, we may choose non-negative $H\in C_0^\infty(\R)$ to be supported in a small enough neighborhood of $0\in \R$ so that the function $f$ defined by (\ref{formula_f_Rn}) satisfies $f=1$ on $\Gamma$. 
Let us define 
  \[
   \tilde{g}=f^{p-2}\eta,
  \]
where $p$ is as in (\ref{def_p}).

Let $u_0 \in C_0^\infty(\Gamma)$ and let $u$ be the solution of (\ref{eq:intro_wave-eq}) with $g = \eta$.  
We define 
  \[
   \tilde{u}=f^{-1}u.
  \]
By using Proposition~\ref{hidden_conformal_invariance}, we see that $\tilde{u}$ solves (\ref{eq:intro_wave-eq}) with $g = \tilde g$.
The boundary condition $\tilde u = u_0$ on $\Sigma$ follows from the facts that $f=1$ on $\Gamma$ 
and that the support of $u_0$ is contained in $\Gamma$.
Note also that 
    \begin{align}\label{aux_th_Rn_init}
\p_t \tilde u = f^{-1} \p_t u - u f^{-2} \p_t f,
    \end{align}
and so $\p_t \tilde u = 0$ when $t = 0$
due to the initial conditions for $u$.   
Moreover, $\tilde g = \eta$ on $\Gamma$, and hence the normal derivative $\p_\nu$ is the same for both the metrics on this set. As $f=1$ in a neighborhood of $\Gamma$, we have
\begin{align*}
\Lambda^{\Gamma,\Gamma}_{\tilde{g}} u_0=\p_{\nu} \tilde{u}|_{\Gamma}=\big(f^{-1}\p_\nu u\big)\big|_{\Gamma}-\big(uf^{-2}\p_\nu f\big)\big|_{\Gamma}=\p_\nu u|_{\Gamma}=\Lambda^{\Gamma,\Gamma}_{\eta} u_0.
\end{align*}
We used $f=1$ on a neighborhood of $\Gamma$ in $[0,T]\times \Omega$ in particular to have that $\p_\nu f=0$.

We have shown that $\Lambda^{\Gamma,\Gamma}_{\tilde{g}}=\Lambda^{\Gamma,\Gamma}_{\eta}$. It follows from Proposition~\ref{lem_conf_triv} that $\tilde g$ and $\eta$ are not isometric unless $f = 1$ identically. As $H$ can be taken to be arbitrarily large near the origin, we see that the family of Lorentzian metrics satisfying (\ref{identical_partial_data_Rn}) is not bounded in any fixed coordinate system. 
\end{proof}

We make here the following remark, which served as a motivation for the proof of Proposition~\ref{lem_conf_triv}. In the setting of Theorem~\ref{thm_Rn}, another proof that $f^{p-2}\eta$ is not, at least generally, induced by a diffeomorphism is achieved as follows. 
We let $f$ be as in~\eqref{formula_f_Rn}, but now we additionally require that $H$ attains a local maximum.
Writing $e^{2h}=f^{p-2}$, we have by Lemma~\ref{transformation_of_Shouten} that
    \begin{align}\label{rem_de_dh}
\frac{\p^2 h}{\p x^k \p x^l}
&= \p_k h\p_l h-m |dh|^2_\eta\s \eta_{kl} + P_{kl}(\eta) - \frac{\p\phi^a}{\p x^k} \frac{\p\phi^b}{\p x^l} {P}_{ab}(\eta) \circ \phi \nonumber
\\&=\p_k h\p_l h-m |dh|^2_\eta\s \eta_{kl}.
    \end{align}
Here we used the fact that the Minkowski space is flat to equate the Schouten tensors to zero. Since $H$ attains a local maximum, it follows that $h$ attains a local maximum. However, this is a contradiction to~\eqref{rem_de_dh} with $k=l=0$, since the Hessian of $h$ is negative definite at the local maximum of $h$. Thus, there is no diffeomorphism $\phi$ such that $\phi^*\eta=f^{p-2}\eta$.

\begin{proof}[Proof of Theorem~\ref{thm_Lorentz}]
As $\Gamma_1 \cup \Gamma_2 \ne \Sigma$, we may choose 
such $\Psi \in C_0^\infty(\Sigma)$
that $\Psi = 0$ on $\Gamma_1 \cup \Gamma_2$
but $\Psi$ does not vanish identically on $\Sigma$.
Let $F \in C^\infty(M)$ be the solution to (\ref{eq:intro_wave-eq_lorentz})
with the boundary value $u_0 = \Psi$.
For small enough $\lambda > 0$, the function 
\[
 f=1+\lambda F,
\]
is positive.
 Let us define 
 \[
\tilde{g}=f^{p-2}g,  
 \]
where $p$ is as in (\ref{def_p}).

Let $u_0 \in C_0^\infty(\Gamma_1)$ and let $u$ be the solution of (\ref{eq:intro_wave-eq_lorentz}).  
We define 
  \[
   \tilde{u}=f^{-1}u.
  \]
As $\Box_g f = 0$, Proposition~\ref{hidden_conformal_invariance} implies that $\tilde{u}$ solves (\ref{eq:intro_wave-eq_lorentz}) with $g$ replaced by $\tilde g$.
Following the proof of Theorem~\ref{thm_Rn},
we see that $\tilde u = u_0$ on $\Sigma$,
$\tilde u$ satisfies the vanishing initial conditions, 
and that the normal derivative $\p_\nu$ is the same for both the metrics $g$ and $\tilde g$ on $\Gamma_2$. 

Let us show that $\p_{\nu} \tilde{u} = \p_\nu u$ on $\Gamma_2$. This argument is different from the proof of Theorem~\ref{thm_Rn}.
Observe that $\tilde u = u_0 = 0$ on $\Gamma_2$ since $u_0 \in C_0^\infty(\Gamma_1)$ and $\Gamma_1 \cap \Gamma_2 = \emptyset$.
Using also the fact that $f=1$ on $\Gamma_2$, we have
\begin{align*}
\p_{\nu} \tilde{u}|_{\Gamma_2}=\big(f^{-1}\p_\nu u\big)\big|_{\Gamma_2}-\big(uf^{-2}\p_\nu f\big)\big|_{\Gamma_2}=\p_\nu u|_{\Gamma_2}.
\end{align*}
We have shown that $\Lambda_{\tilde{g}}^{\Gamma_1,\Gamma_2} = \Lambda_{g}^{\Gamma_1,\Gamma_2}$,
and we conclude the proof by using Proposition~\ref{lem_conf_triv}.
 \end{proof}

\appendix
\section{Further proofs}

\begin{Lemma}\label{lem_timeindep}
Let $M$ be as in (\ref{def_M}) with large $T > 0$ in comparison to the diameter of $\Omega$.
Let $\Gamma_0 \subset \p \Omega$ be open and nonempty and set $\Gamma = (0,T) \times \Gamma_0$.
Let $c \in C^\infty(\Omega)$ be strictly positive
and suppose that $c = 1$ near $\Gamma_0$. 
Then 
$\Lambda_{c\eta}^{\Gamma,\Gamma}$
determines $c$ uniquely.
\end{Lemma}
\begin{proof}
If $\Box_{c \eta} u = 0$ then, writing 
    \begin{align*}
c = f^{p-2}, \quad v = fu, \quad q = -f^{-1}\Box_\eta f = f^{-1}\Delta f,
    \end{align*}
there holds $\Box_\eta v + q v = 0$ due to Proposition \ref{hidden_conformal_invariance}. 
As $c = 1$ near $\Gamma$, $u=v$ and $\p_\nu u = \p_\nu v$ on $\Gamma$, and the map $\Lambda_{c\eta}^{\Gamma,\Gamma}$ determines $q$ by using the Boundary Control method, see e.g. \cite{KKL}. 
Moreover, as $f=1$ near $\Gamma$ and $\Delta f - qf = 0$,
we see using elliptic unique continuation that $f$ is determined uniquely. The same holds for $c = f^{p-2}$.
\end{proof}

The above proof shows that the wave equation with a conformally scaled metric can be reduced to a Shr\"odinger equation. Therefore, one might ask if our counterexamples for the inverse problem for the wave equation yield counterexamples to the inverse problem for the Shr\"odinger equation. This is not case. The wave equation $\square_{cg}u=0$ is conformally equivalent to 
the Shr\"odinger equation
\[
 \square_{g}u+ qu=0
\]
with $q=-f^{-1}(\square_g f)$ and $c = f^{p-2}$. However, since in our counterexamples $f$ solves $\square_gf=0$, we have $q=0$. While we have shown that wave equations associated to two different but conformal metrics might have the same DN maps, this does not imply that there are two different Shr\"odinger equations, which have identical DN maps.

\begin{Lemma}\label{lem_strict_hyper}
Suppose that $\tau \in C^\infty(\mathcal M)$ is temporal.
Then $\Box_g$ is strictly hyperbolic with respect to the level surfaces of $\tau$ in the sense of \cite[Def. 23.2.3]{H3}.
\end{Lemma}
\begin{proof}
The principal symbol $p(x,\xi) = (\xi,\xi)$, $(x,\xi) \in T^* M$, of $\Box_g$ satisfies
    \begin{align*}
p(x, d\tau(x)) = g(d\tau, d\tau) < 0.
    \end{align*}
Moreover, for any $\xi \in T_x^* M$, linearly independent from $d\tau(x)$, the polynomial
    \begin{align*}
p(x, \xi + \lambda d\tau(x))
= a \lambda^2 + b \lambda + c
    \end{align*}
where $a = g(d\tau, d\tau) < 0$, $b = 2 g(\xi, d\tau)$ and $c = g(\xi,\xi)$, has two real roots. To see this we write $\xi = z d \tau + \xi^\perp$ where $g(\xi^\perp, d\tau) = 0$ and $\xi^\perp \ne 0$.
Then $\xi^\perp$ is spacelike and
    \begin{align*}
b^2 - 4 a c
= 4 a^2 z^2 - 4 a (a z^2 + g(\xi^\perp, \xi^\perp))
= - 4 a g(\xi^\perp, \xi^\perp) > 0.
    \end{align*}
\end{proof}

\bibliographystyle{abbrv}
\bibliography{refs}
\end{document}